\documentclass[authoryear]{elsarticle}

\usepackage{graphicx}
\usepackage{epstopdf}
\usepackage{amsfonts}
\usepackage{color}
\usepackage{mathrsfs}
\usepackage{psfrag, amssymb, amsmath, cite}
\usepackage{verbatim}
\usepackage{epsfig} 
\usepackage{times} 
\usepackage{amsmath} 
\usepackage{amssymb}  
\usepackage{amsfonts}
\usepackage[colorlinks,linkcolor=red,anchorcolor=blue,citecolor=blue]{hyperref}
\usepackage{cite}
\usepackage[ruled]{algorithm2e}
\usepackage{fancybox}
\usepackage{framed}
\usepackage{enumitem}

\usepackage{amsthm}
\usepackage{mathtools}

\newtheorem{remark}{Remark}

\newtheorem{theorem}{\textbf{Theorem}}

\newtheorem{lemma}{\textbf{Lemma}}

\renewcommand\footnoterule{\kern-3pt \hrule width 2in \kern 2.6pt}

\journal{Nowhere}

\begin{document}

\begin{frontmatter}



\title{\LARGE A note on real   similarity \\ to a diagonal dominant matrix}


\author{Zhiyong Sun, Brian D. O. Anderson, Wei Chen}
\address[Tue]{Department of Electrical Engineering, Technische Universiteit Eindhoven, the Netherlands}
\address[ANU]{School of Engineering, The Australian National University, Canberra ACT
2601, Australia}
\address[PKU]{Department of Mechanics and Engineering Science, Peking University, Beijing 100871, China}


\begin{abstract}
This note  presents several conditions to characterize \textbf{real  matrix similarity} between a Hurwitz matrix (and then more generally,  a real square matrix) and a \textbf{diagonal dominant matrix}. 
\end{abstract}

\begin{keyword}


Hurwitz matrix, diagonal dominant matrix, real similarity. 
\end{keyword}

\end{frontmatter}


\section{Definitions and notations}
\begin{itemize}
    \item A real square matrix $A  = \{a_{ij}\}\in \mathbb{R}^{n \times n}$ is  \textbf{Hurwitz}  if all its eigenvalues have negative real parts. 
    \item \textit{Row}-diagonal dominant matrix:
    \begin{itemize}
        \item A real matrix $A  = \{a_{ij}\}\in \mathbb{R}^{n \times n}$ is (non-strict)  \textbf{\textit{\textbf{row}}-diagonal dominant}, if the absolute value of each diagonal entry is greater than or equal to the sum of the absolute values of off-diagonal entries in that row; \\i.e., $|a_{ii}|  \geq \sum_{j=1, j \neq i}^{n} |a_{ij}|, \forall i = 1, 2, \cdots, n$. 
    \item If the absolute value of each diagonal entry is \textbf{strictly} greater than  the sum of the absolute values of off-diagonal entries in that row, i.e., $
    |a_{ii}| > \sum_{j=1, j \neq i}^{n} |a_{ij}|, \forall i = 1, 2, \cdots, n$, then we call the matrix $A$ \textbf{strictly row-diagonal dominant}. 
    \end{itemize}
\item Similarly, we can define a \textbf{column-diagonal dominant} matrix by following the above definition but using the column sums of  off-diagonal entries. In this note, we call a matrix $A$ \textbf{diagonal dominant} if it is either row-diagonal dominant or column-diagonal dominant. \footnote{For the case of a complex matrix, the definition of a diagonal dominant complex matrix follows similarly but replacing `absolute value' by `magnitude' in the definition.   }

    \item A  real square matrix is a \textbf{Z-matrix} if it has nonpositive off-diagonal elements. 
    \item A \textbf{Metzler} matrix is a real matrix in which all the off-diagonal components are nonnegative (equal to or greater than zero). 
    \item A matrix $A \in \mathbb{R}^{n \times n}$ is called an \textbf{M-matrix}, if it is a Z-matrix and its  eigenvalues have positive real parts.  \footnote{In this note by convention an M-matrix is meant a \textit{non-singular} M-matrix. We will make it clear to distinguish non-singular M-matrix and \textit{singular M-matrix} in the context (A typical example of a singular M-matrix is   graph Laplacian matrix).}

\item For a real matrix $A = \{a_{ij}\} \in \mathbb{R}^{n \times n}$, we associate it with a  \textbf{comparison matrix} $M_A = \{m_{ij}\} \in \mathbb{R}^{n \times n}$, defined by 
    \begin{align}
    m_{ij} =   \left\{
       \begin{array}{cc}
       |a_{ij}|,  &\text{  if  } \,\,\,\,j  = i;  \\ \nonumber
       -|a_{ij}|,  &\text{  if  } \,\,\,\, j  \neq i.   \nonumber  
       \end{array}
      \right.
    \end{align}
    A given matrix $A$ is called an \textbf{H-matrix} if its comparison matrix $M_A$ is an M-matrix.  
\end{itemize}

\section{Problem of interest}
In this note, we are interested in the following problem:
\begin{itemize}
    \item When is a real Hurwitz matrix $A$ real similar to a real diagonal dominant matrix $B$?
\end{itemize}
and a more general problem on real similarity:
\begin{itemize}
    \item When is a real square matrix $A$ real similar to a   diagonal dominant matrix $B$?
\end{itemize}

By ``a real matrix $A$ is real similar to a real matrix $B$" we mean there exists a \textbf{real} non-singular matrix $P$ such that $B = PAP^{-1}$.

\section{Real similarity with $2 \times 2$ real matrix} 
\subsection{Case I: real matrix with real eigenvalues}
We first give the following conditions to characterize real similarity for $2 \times 2$ real Hurwitz matrix with real eigenvalues. 
\begin{lemma}  \label{lemma:lreal_eigenvalue}
Consider a $2 \times 2$ real Hurwitz matrix $A$ and suppose its two eigenvalues are real, denoted by $\lambda_1 <0$ and $\lambda_2<0$. Then there exists a real non-singular matrix $P_1$ such that $A$ is real similar to a real matrix $B_1 = P A P^{-1}$:
\begin{align}
    B_1=\begin{bmatrix} \lambda_1 &\star \\ 0 & \lambda_2
\end{bmatrix} 
\end{align}
where $\star$ denotes a `don't-care' term. Furthermore, there exists a real non-singular matrix $P_2$ such that $A$ is real similar to a  strictly diagonal dominant matrix $B_2 = P_2A P_2^{-1}$:
\begin{align}
    B_2=\begin{bmatrix} \lambda_1 & \beta \\ 0 & \lambda_2
\end{bmatrix} 
\end{align}
where $|\lambda_1| > |\beta|$. 
\end{lemma} 
\begin{proof}
The proof of the first statement follows similarly the standard procedure of real Jordan decomposition (e.g., \citep{horn2012matrix}) and is omitted here. We proceed with the proof of the second statement, which is equivalent to proving that `$B_1$ is real similar to $B_2$'. Consider a real diagonal matrix $\bar P = \text{diag} \{k_1, k_2\}$ with $k_1 >0, k_2 >0$. Then 
\begin{align}
    \bar P B_1  \bar P^{-1}=\begin{bmatrix} \lambda_1 &   (k_1 \star)/k_2 \\ 0 & \lambda_2
\end{bmatrix} 
\end{align}
Since both $\lambda$ and the `don't-care' term $\star$ are  bounded, there exist $k_1$ and $k_2$ such that $ |\lambda_1| > |(k_1 \star)/k_2| : = \beta$. By letting $P_2 := \bar P P$, we complete the proof.  
\end{proof}
Note  the proof of the above lemma does not require   the condition that $A$ is Hurwitz. The statements still hold with the condition  that the matrix $A$ has two real and \textbf{non-zero} eigenvalues. In fact, we conclude:
\begin{itemize}
    \item Every real $2 \times 2$   matrix with non-zero real eigenvalues is real similar to a strictly diagonal dominant real matrix. 
\end{itemize}

\subsection{Case II: real matrix with complex eigenvalues}
We then consider conditions for real matrix with complex eigenvalues. To begin, the Hurwitz constraint will not be imposed. 
\begin{lemma} \label{lemma:real_2by2_complex}
Consider the following $2 \times 2$ real matrices
\begin{equation}\label{eq:J}
J=\begin{bmatrix} \alpha&\beta\\-\beta&\alpha
\end{bmatrix} \quad \beta\neq 0
\end{equation}
and
\begin{equation}\label{eq:K}
K=\begin{bmatrix} a&b\\c&d
\end{bmatrix}
\end{equation}
Then $J$ and $K$ are similar if and only if there exist real $x$ and $y\neq 0$ such that
\begin{align}\label{eq:xy}
a&=\alpha-x\\\nonumber
d&=\alpha+x\\\nonumber
b&=y^{-1}\sqrt{\beta^2+x^2}\\\nonumber
c&=-y\sqrt{\beta^2+x^2}
\end{align}
\end{lemma}

\begin{proof}
Suppose first that $J$ and $x,y\neq 0$ are given and the quantities $a,b,c,d $ are then defined as above. It is easily seen that $a+d=2\alpha$ and $ad-bc=\alpha^2+\beta^2$, ensuring that $J$ and $K$ have the same trace and determinant. Hence they are similar. 

Conversely, suppose $J$ and $K$ are given and known to be similar. Since the traces must be equal, there holds $a+d=2\alpha$, and then there must be an $x$ satisfying the first two equations of \eqref{eq:xy}; we can define it as $\alpha-a$ and there follows $x=d-\alpha$. 

We next argue that $b$ and $c$ must be nonzero. For a contradiction, we suppose to the contrary. Then the determinant of $K$ would be $\alpha^2-x^2$ while the determinant of $J$ is $\alpha^2+\beta^2$. Equality of the two determinants would imply that $x=0,\beta=0$, which contradicts the theorem hypothesis that $\beta\neq 0$.

Next, define $y$ by $y=\sqrt{\beta^2+x^2}/b$, and observe that $y\neq 0$. Using again the determinantal equality $\alpha^2+\beta^2=ad-bc$, and the expressions for $a,d$ and $b$ in \eqref{eq:xy}, the equality for $c$ in \eqref{eq:xy} is immediate.
\end{proof}

Given a matrix $J$ as in \eqref{eq:J} with $|\alpha|<|\beta|$, we now show that there is no choice of $x,y$ for which we can even get non-strict diagonal dominance of $K$. 
\begin{lemma} \label{lemma:impossible}
Consider the matrix $J$ of \eqref{eq:J} with $|\alpha|<|\beta|$. Then there is no choice of $x,y\neq 0$ such that the matrix $L :=K$ as below
\begin{equation}\label{eq:L}
L=\begin{bmatrix}
\alpha-x&y^{-1}\sqrt{\beta^2+x^2}\\
-y\sqrt{\beta^2+x^2}&\alpha+x
\end{bmatrix}
\end{equation}
is non-strictly diagonal dominant, and therefore no non-singular similarity transformation of $J$ for which the transformed matrix is non-strictly diagonal dominant. 
\end{lemma}
\begin{proof}
Suppose, to obtain a contradiction, that $|\alpha|<|\beta|$  and there is a choice of $x,y\neq 0$ ensuring that $L$ is non-strictly diagonal dominant. 
Then there holds
\begin{align}\label{eq:DD}
|\alpha-x|&\geq|y^{-1}|\sqrt{\beta^2+x^2}\\\nonumber
|\alpha+x|&\geq|y|\sqrt{\beta^2+x^2}
\end{align}
 Multiplying the two inequalities together gives
 \begin{equation}
 |\alpha^2-x^2|=|\alpha-x||\alpha+x|\geq\beta^2+x^2
 \end{equation}
 Since $\beta^2>\alpha^2$, there exists a nonzero $\gamma^2$ such that $\beta^2=\alpha^2+\gamma^2$, and then the above inequality becomes
 \begin{equation}\label{eq:a_and_x}
 |\alpha^2-x^2|\geq\alpha^2+\gamma^2+x^2
 \end{equation}
 However, this inequality can never be satisfied: the inequality \eqref{eq:a_and_x} implies $ \alpha^2 + x^2 \geq  |\alpha^2-x^2|\geq\alpha^2+\gamma^2+x^2$ which is impossible with $\gamma^2 >0$. Therefore, there does not exist a non-singular  real transformation of $J$ for which the transformed matrix is non-strictly diagonal dominant. This completes the proof. 
\end{proof} 

We now summarize the above results and present the following main conditions.  
\begin{theorem} \label{thm:2by2}
Consider a $2 \times 2$ real Hurwitz matrix $A$ and suppose its two eigenvalues are complex, denoted by $\lambda_{1,2}  = \alpha \pm \beta j$ with   $\alpha<0, \beta\neq 0$.  Then there exists a real non-singular matrix $P_1$ such that $A$ is real similar to a real matrix $B_1$
\begin{align} \label{eq:B_1_real_complex}
    B_1=\begin{bmatrix} \alpha & \beta \\ -\beta & \alpha
\end{bmatrix} 
\end{align}
Furthermore, 
\begin{itemize}
    \item If $|\alpha| > |\beta|$, then there exists a real non-singular matrix $P$ such that $A$ is real similar to a \textbf{strictly} diagonal dominant real matrix $B$. 
    \item If $|\alpha| = |\beta|$, then there exists a real non-singular matrix $P$ such that $A$ is real similar to a \textbf{non-strict} diagonal dominant real matrix $B$. 
    \item If  $|\alpha| < |\beta|$, then there \textbf{does not} exist a real non-singular matrix $P$ such that $A$ is real similar to a  diagonal dominant real matrix $B$. 
\end{itemize}
\end{theorem}

\begin{proof}
The proof of the first statement follows similarly the standard procedure of real Jordan decomposition (e.g., \citep{horn2012matrix}) and is omitted here. We proceed with the proof of the second set of three statements.  If $|\alpha| > |\beta|$, then the matrix $B_1$ in \eqref{eq:B_1_real_complex} is a \textbf{strictly} diagonal dominant real matrix.  If $|\alpha| = |\beta|$, then the matrix $B_1$ in \eqref{eq:B_1_real_complex} is a \textbf{non-strict} diagonal dominant real matrix. By letting $B_1 : = B$ we conclude the  statements of these two cases. For the third case that $|\alpha| < |\beta|$, Lemma~\ref{lemma:impossible} states that there is no non-singular real transformation matrix $P$ such that $A$ is real similar to a  diagonal dominant real matrix $B$. 
\end{proof}

Note the above lemmas and Theorem~\ref{thm:2by2} do not require the condition that $A$ is Hurwitz. The statements on real similarity in Theorem~\ref{thm:2by2} still hold under the condition that the complex eigenvalues of the matrix $A$ satisfy $|\alpha| > |\beta|$. In particular, we conclude:
\begin{itemize}
    \item Every real $2 \times 2$  matrix with complex eigenvalues $\lambda_{1,2}  = \alpha \pm \beta j$ and with   $|\alpha| > |\beta|$ is real similar to a strictly diagonal dominant real matrix. 
\end{itemize}

As a side note, we also consider the problem of complex similarity of a $2 \times 2$ real matrix to a complex diagonal dominant matrix. 
\begin{lemma} \label{lemma:complex_similar}
   Every   $2 \times 2$ real Hurwitz  matrix is complex similar to a strictly diagonal dominant complex matrix; i.e., for any $2 \times 2$   real Hurwitz matrix $A$, there exists a \textbf{complex} non-singular matrix, such that the real matrix $A$ is similar to a (possibly complex) strictly diagonal dominant matrix.
\end{lemma}
\begin{proof}
    For the case that  the Hurwitz matrix $A$ has two real eigenvalues, the proof follows similarly to that of Lemma~\ref{lemma:lreal_eigenvalue} while we can take the real matrix $P$ as the transformation matrix. For the case that  the Hurwitz matrix $A$ has complex eigenvalues, it is real similar to the real matrix $B_1$ in \eqref{eq:B_1_real_complex}. Taking a non-singular complex matrix
\begin{align}
    P_3 =    \begin{bmatrix} -j & -j \\ 1 & -1
\end{bmatrix} 
\end{align}
we can obtain
\begin{align}
   P_3^{-1} B_1 P_3 =  \begin{bmatrix} \alpha + b j & 0 \\ 0 & \alpha - b j
\end{bmatrix} 
\end{align}
which is complex and diagonal. This completes the proof. 
\end{proof}
Note: the statement in this lemma (for the complex eigenvalue  case) is a special case of the general result: a matrix that has distinct eigenvalues is diagonalizable. 

\begin{remark}
The above results treat similarity to a row-diagonal dominant matrix; while evidently, all main statements can be extended to a column-diagonal dominant matrix. This is due to the following well-known result: a real square matrix $A$ is real similar to $A^T$.  
\end{remark}
 \section{Intuitive interpretation by Gershgorin's  circles }
The Gershgorin circle theorem \citep{varga2010gervsgorin} gives a convenient method for characterizing locations of eigenvalues for a (real or complex) matrix: all eigenvalues of a matrix are located in the union of all Gershgorin circles, while each circle is centered at each diagonal entry with radius defined by the sum of the magnitude of all off-diagonal entries (in the corresponding row or column). We now provide a Gershgorin circle interpretation to the main results of Lemmas 1-4. 

For Lemma~\ref{lemma:lreal_eigenvalue}  when the Hurwitz matrix $A$ has real non-zero eigenvalues: the real matrix $P$ serves as a real scaling matrix such that the transformed real matrix $B_2 = P_2^{-1} A P_2$ has two scaled Gershgorin circles; each circle is centered at one real eigenvalue and does not touch upon the origin point. In fact, the radius of each circle can be made arbitrarily small by a carefully chosen transformation matrix $P$, and the diagonal dominant property of $B_2$ is ensured. 

For Theorem~\ref{thm:2by2} when the matrix $A$ has complex eigenvalues and is real similar to the matrix the matrix $B_1$ in \eqref{eq:B_1_real_complex}: the real similarity transformation via $P$ results in another real matrix (in particular, the diagonal entries are real), and the centers  of all Gershgorin circles are restricted to lying on the real line. For the matrix  $B_1$ in \eqref{eq:B_1_real_complex} and its real similarity classes, the \textbf{smallest} Gershgorin circle  with a center on the real line that includes both complex eigenvalues is the circle centered at $\alpha$ with a radius $|\beta|$. For $|\beta| > |\alpha|$, such a Gershgorin circle encircles the origin and always touches upon the complex line. There is no hope of finding a real transformation matrix $P_2$ to further scale the circles so that they do  not encircle the origin, which implies that $A$ cannot be real similar to a diagonal dominant matrix. 

For Lemma~\ref{lemma:complex_similar} with complex matrix similarity: when the transformation matrix $P$ is allowed to be a complex matrix, the transformed matrix can also be complex (in particular, the diagonal entries can be complex). The Gershgorin circles are not necessarily centered on the real line but can be relocated in the complex plane by the complex matrix $P$. In this case, when the eigenvalues have a non-zero real part, each  Gershgorin circle  can be shifted and centered at an eigenvalue, and its size can be rescaled with an arbitrarily small radius so that it does not encircle the origin. Thus, in this case, any real matrix with each eigenvalue of non-zero real part is complex similar to a (possibly complex) diagonal dominant matrix.

\newpage
\section{Real similarity with  general matrix}
We are now ready to show the following general results. We note the following theorem does not require the Hurwitz condition. 
\begin{theorem} \label{thm:general}
    Consider a real square non-singular matrix $A \in \mathbb{R}^{n \times n}$. Then the following statements hold.
    \begin{itemize}
        \item Suppose that all eigenvalues of $A$ are real. Then $A$ is real similar to a strictly diagonal dominant real matrix.  
        \item Suppose that some eigenvalues of $A$ are complex and for all complex eigenvalue pairs $\lambda_i = \alpha_i \pm \beta_i j$ there holds $|\alpha_i| > |\beta_i|, \forall i$. Then  $A$ is  real similar to a strictly diagonal dominant real matrix.  
        \item Suppose that some eigenvalues of $A$ are complex and for all complex eigenvalue pairs $\lambda_i = \alpha_i \pm \beta_i j$ there holds $|\alpha_i| \geq |\beta_i|, \forall i$. Furthermore, for any complex eigenvalue pair $\lambda_i = \alpha_i \pm \beta_i j$ with $|\alpha_i| = |\beta_i|$,  suppose its  geometric  multiplicity equals its algebraic multiplicity. Then  $A$ is real similar to a non-strict diagonal dominant real matrix.  
        \item For all other cases not covered by the above conditions, there \textbf{does not} exist a real non-singular matrix $P$ such that $A$ is real similar to a  diagonal dominant real matrix. 
    \end{itemize}
    \end{theorem}
    \begin{proof}
      First, we recall the fact that any real square matrix $A \in \mathbb{R}^{n \times n}$ is real similar to the following real Jordan normal form (cf. \citep[Chapter 3]{horn2012matrix})
      \begin{align}
          J = PA P^{-1} = \text{blk-diag} \{J_1, J_2, \cdots, J_i, \cdots, J_k\}
      \end{align}
      where $P$ is a real non-singular matrix, $J_i$ is the $i$-th \textbf{real} Jordan block associated with the eigenvalues of $A$.
          \begin{itemize}
        \item If all eigenvalues of $A$ are real, then the diagonal entries of $J$ are the set of eigenvalues of $A$ (counting multiplicities). Following a similar proof to that of Lemma~\ref{lemma:lreal_eigenvalue}, one can find a real non-singular diagonal matrix $\bar P$ such that $\bar P J \bar P^{-1}$ is strictly diagonal dominant.  
        \item  If some eigenvalues of $A$ are complex, denoted by $\lambda_i = \alpha_i \pm \beta_i j$, then the $2 \times 2$ block matrix 
\begin{align}  
    B_i=\begin{bmatrix} \alpha_i & \beta_i \\ -\beta_i & \alpha_i
\end{bmatrix}
\end{align}        
   appears in the $i$-th real Jordan block that takes the following form
\begin{align} \label{eq:jordan_block}
J_{i}={\begin{bmatrix}B_{i}&I_2&&\\&B_{i}&\ddots &\\&&\ddots &I_2\\&&&B_{i}\end{bmatrix}}
\end{align}  
We now consider the following cases that  for all complex eigenvalue pairs $\lambda_i = \alpha_i \pm \beta_i j$ there hold $|\alpha_i| \geq |\beta_i|, \forall i$.
 \begin{itemize}
     \item Case I:  for all complex eigenvalue pairs $\lambda_i = \alpha_i \pm \beta_i j$ there hold $|\alpha_i| > |\beta_i|, \forall i$. Then following a similar proof to that of Lemma~\ref{lemma:lreal_eigenvalue}, one can find a real non-singular  matrix $\bar P_i$ such that $\bar P_i J_i \bar P_i^{-1}$ is strictly diagonal dominant, which in turn ensures that $A$ is real similar to a strictly diagonal dominant real matrix.  
        \item Case II:  for all complex eigenvalue pairs $\lambda_i = \alpha_i \pm \beta_i j$ with $|\alpha_i| = |\beta_i|$,  suppose the  geometric  multiplicity equals the algebraic multiplicity. Then the real Jordan block $J_i$ for such an eigenvalue pair $\lambda_i$ with $|\alpha_i| = |\beta_i|$  degenerates to 
        $J_i = B_i$. According to Theorem~\ref{thm:2by2},  this ensures that $A$ is real similar to a non-strict diagonal dominant real matrix.  
        \item Case III: for all complex eigenvalue pairs $\lambda_i = \alpha_i \pm \beta_i j$ with $|\alpha_i| = |\beta_i|$,  suppose for at least one such pair, the  algebraic multiplicity is greater than the geometric  multiplicity. Then   the real Jordan block $J_i$ in \eqref{eq:jordan_block} for such an eigenvalue pair $\lambda_i$ has at least one off-diagonal identity block $I_2$. As a consequence, there \textbf{does not} exist a real non-singular matrix $\bar P_i$ such that $\bar P_i J_i \bar P_i^{-1}$ is  diagonal dominant, and $A$ is not real similar to a diagonal dominant matrix. 
    \end{itemize}
\item For all other cases with complex eigenvalue pairs   (i.e., there exist at least one complex eigenvalue with $|\alpha_i| < |\beta_i|$), according to Theorem~\ref{thm:2by2}, the associated real Jordan block $J_i$ cannot be made to be real similar to a diagonal dominant matrix. As a consequence, there \textbf{does not} exist a real non-singular matrix $P$ such that $A$ is real similar to a  diagonal dominant real matrix. 
     \end{itemize}
 
    \end{proof}

To complete the story, we further discuss complex similarity of a general real square matrix
to a  diagonal dominant matrix.
\begin{theorem}
    Every real Hurwitz matrix is complex similar to a (possibly complex) strictly diagonal dominant   matrix.  
\end{theorem}    
A more general result on complex similarity is shown below.
\begin{theorem}
    Every real non-singular matrix  is complex similar to a (possibly complex) strictly diagonal dominant   matrix.  
\end{theorem}  
The proofs follow similarly the proof of the $2  \times 2$ matrix case and the real similarity in Theorem~\ref{thm:general}, which are omitted here. 

\section{Special cases}
Some special cases are recalled here for further references \citep{horn2012matrix, berman1994nonnegative}. 


\begin{lemma}
    Consider a real Hurwitz matrix $A$ and suppose its off-diagonal entries are non-negative (i.e., $-A$ is an M-matrix). Then there exists a real \textbf{diagonal} matrix $K$ such that $K A K^{-1}$ is strictly diagonal dominant with negative diagonal entries. 
\end{lemma}

\begin{lemma}
    Consider a real Hurwitz matrix and suppose it is also an H-matrix.  Then there exists a real \textbf{\textbf{diagonal}} matrix $K$ such that $K A K^{-1}$ is strictly diagonal dominant with negative diagonal entries.
\end{lemma}


\newpage
\bibliography{Hurwitz}
\bibliographystyle{elsarticle-harv}

\end{document}